\documentclass{amsart}
\usepackage{amsfonts}
\usepackage{amssymb}
\usepackage{amsmath}
\usepackage[dvipdf]{graphicx}
\newtheorem{theorem}{Theorem}
\newtheorem{proposition}{Proposition}

\newtheorem{lemma}{Lemma}
\newtheorem{definition}{\sc{Definition}}

\title[Translation Hypersurfaces]{Translation Hypersurfaces with Constant $S_r$
Curvature in the Euclidean Space}

\author[B. P. Lima]{B. P. Lima$^1$}
\author[N. L. Santos]{N. L. Santos$^2$}
\author[J. P. Silva]{J. P. Silva$^3$}
\author[P. Sousa]{P. Sousa$^4$}
\address{Universidade Federal do Piau\'{i}; Departamento de Matem\'{a}tica; 64049-550,
Ininga - Teresina - PI; Brazil}
\email{barnabe@ufpi.edu.br$^1$}
\email{newtonls@ufpi.edu.br$^2$}
\email{jsilva@ufpi.edu.br$^3$}
\email{paulosousa@ufpi.edu.br$^4$}



\thanks{The first, second and third author are partially supported by CAPES/Brazil. The fourth author is partially supported by CAPES and CNPq/Brazil}

\begin{document}
\maketitle

\begin{abstract}
In this paper, we give a complete description of all translation hypersurfaces with constant $r$-curvature $S_r$,
in the Euclidean space.
\end{abstract}

\section{Introduction and Statement of Results}

In the Euclidean space $\Bbb{R}^{3}$, a surface $M^{2}$ is called a translation
surface if it is given by an immersion
\[
\Psi:U\subset\Bbb{R}^{2}\to\Bbb{R}^{3}:\, \Psi(x,y)=(x,y,z(x,y))
\]
where $z(x,y)=f(x)+g(y)$, for $f$ and $g$ smooth functions of a single variable,
that is, $\Psi$ is obtained as an Euclidean translation of the smooth curve,
$\alpha(x)=(x,0,f(x))$, pointwisely along the curve $\beta (y)=(0,y,g(y))$.
Scherk \cite{Scherk} proved in 1835 that, besides the planes, the only minimal
translation surfaces are the surfaces given by
\[
z(x,y)=\frac{1}{a}\ln\left|\frac{\cos(ay)}{\cos(ax)}\right|
\]
where $a$ is a nonzero constant. This surface, unique up to similarities, is
called Scherk's surface. The concept of translation surfaces was generalized
to hypersurfaces of $\Bbb{R}^{n+1}$ by Dillen, Verstraelen and Zafindratafa
\cite{Dillen}, they obtained a classification of minimal translation
hypersurfaces of the $(n+1)$-dimensional Euclidean space.

\begin{definition}
We say that a hypersurface $M^{n}$ of the Euclidean space $\Bbb{R}^{n+1}$ is a
translation hypersurface if it is the graph of a function given by
\[
 F(x_{1},\ldots,x_{n})=f_1(x_{1})+\ldots+f_{n}(x_{n})
\]
where $(x_{1},\ldots,x_{n})$ are cartesian coordinates and $f_{i}$ is a smooth
function of one real variable for $i=1,\ldots,n$.
\end{definition}

That is, $M^n$ can be thought as a composition of plane curves given by graphs,
that is: denote by $\alpha_i (t_i)=t_ie_i+f_i(t_i)e_{n+1}$, for $i=1,\ldots ,n$,
a family of plane curves, actually, graphs. For $p\in \mathbb R^{n+1}$ denote by
$L_p:\mathbb R^{n+1}\to \mathbb R^{n+1}$, the translation through $p$, given by
$L_p(q)=p+q$. Then, the map $\psi$ above is given by $$(x_1,\ldots x_n)\mapsto
L_{\alpha_1(x_1)}\circ \ldots \circ L_{\alpha_{n-1}(x_{n-1})}(\alpha_n(x_n)).$$

Another extension was obtained by R. L\'{o}pez, in \cite{Lopez}, where it is
introduced the concept of translation surfaces in the $3$-dimensional hyperbolic
space and it is presented a classification of minimal translation surfaces.

In a different aspect, Liu \cite{Liu} considered the translation surfaces with
constant mean curvature in $3$-dimensional Euclidean space and Lorentz-Minkowski
space. A classification of translation hypersurfaces with constant mean
curvature in $(n+1)$-dimensional Euclidean space was made by Chen, Sun and Tang
\cite{CST}.

Now, let $M^{n}\subset\mathbb{R}^{n+1}$ be an oriented hypersurface
and $\lambda_{1},\ldots,\lambda_{n}$ denote the principal curvatures of $M^{n}$.
We can consider similar problems related with the $r$-th elementary symmetric
polynomials, $S_{r}$, given by $S_{r}=\sum\lambda_{i_{1}}\cdots\lambda_{i_{r}}$,
where $r=1,\ldots,n$ and $1\leq i_{1}<\cdots<i_{r}\leq n$. In particular, $S_1$
is the mean curvature, $S_2$ the scalar curvature and $S_n$ the Gauss-Kronecker
curvature, up to normalization factors. A very useful relationship involving the
various $S_r$ is given by the next proposition, (Proposition 1, from
\cite{Caminha}).  This result will play a central role along this paper.

\begin{proposition}[Caminha, 2006  {\cite{Caminha}}] \label{Caminha}
Let $n>1$ be an integer, and $\lambda_1, \ldots , \lambda_n$ be real numbers.
Define, for $0\leq r\leq n$, $S_r=S_r(\lambda_1,\ldots, \lambda_n)$ as above, and set
$H_r=H_r(\lambda_1,\ldots, \lambda_n)=\large\binom{n}{r}^{-1}S_r(\lambda_1,\ldots, \lambda_n)$

\vspace{.2cm}

\noindent $(a)$ For $1\leq r\leq n$, one has $H_r^2\ge H_{r-1}H_{r+1}$.
Moreover, if equality happens for $r=1$ or for some $1<r<n$, with $H_{r+1}\neq
0$ in this case, then $\lambda_1=\ldots =\lambda_n $.

\vspace{.2cm}

\noindent $(b)$ If $H_1,H_2, \ldots H_r>0$ for some $1<r\leq n$, then $$H_1\ge
\sqrt{H_2}\ge \sqrt[3]{H_3}\ge \cdots \ge \sqrt[r]{H_r}\,.$$ Moreover, if
equality happens for some $1\leq j< r $, then $\lambda_1=\ldots =\lambda_n$.

\vspace{.2cm}

\noindent $(c)$ If for some $1\leq r<n$, one has $H_r=H_{r+1}=0$, then $H_j=0$
for all $r\leq j \leq n$. In particular, at most $r-1$ of the $\lambda_i$ are
different from zero.
\end{proposition}

\vspace{.3cm}

In \cite{MLeite}, M. L. Leite gave a new example of a translation
hypersurface of $\Bbb{R}^{4}$ with zero scalar curvature. And, in
\cite{LSS}, Lima, Santos and Sousa presented a complete description of all
translation hypersurfaces with zero scalar curvature in the Euclidean space
$\Bbb{R}^{n+1}$. In this paper, we obtain a complete classification of
translation hypersurfaces of $\Bbb{R}^{n+1}$ with $S_r=0$. We prove the following

\begin{theorem}\label{th1}
Let $M^{n}\,(n\geq3)$ be a translation hypersurface in $\Bbb{R}^{n+1}$. Then, for $2<r<n$,
$M^{n}$ has zero $S_r$ curvature if, and only if, it is congruent to
the graph of the following functions
\begin{itemize}
\item $F(x_{1},\ldots,x_{n})=\displaystyle\sum_{i=1}^{n-r+1}a_{i}x_{i}
+\sum_{j=n-r+2}^{n}f_{j}(x_{j})+b$,
\end{itemize} on $\Bbb R^{n-r+1}\times J_{n-r+2}\times\cdots\times
J_n$, for some intervals $J_{n-r+2},\ldots,J_{n}$, and arbitrary smooth functions $f_{i}:J_i\subset \Bbb{R}\to\Bbb{R}$. Which defines, after a suitable linear change of variables, a vertical cylinder, and
\begin{itemize}
\item A generalized periodic Enneper hypersurface given by
 \begin{eqnarray*}
F(x_{1},\ldots,x_{n})&=&\displaystyle\sum_{i=1}^{n-r-1}a_{i}x_{i}\\
&+&\sum_{k=n-r}^{n-1}\frac{\sqrt{\beta}}{a_k}\ln\left|\dfrac{\cos
\left(-\dfrac{a_{n-r}\ldots a_{n-1}}{\sigma_{r-1}(a_{n-r},\ldots ,
a_{n-1})}\sqrt{\beta}x_n+b_n\right)}{\cos(a_k\sqrt{\beta}x_k+b_k)}\right|+c
\end{eqnarray*}
\end{itemize}
on $\Bbb R^{n-r-1}\times I_{n-r}\times\cdots\times I_{n}$, where $a_1, \ldots,
a_{n-r},\ldots ,a_{n-1},b_{n-r},\ldots , b_{n}$ and $c$ are real constants
where $a_{n-r}, \ldots, a_{n-1}$ and $\sigma_{r-1}(a_{n-r}, \ldots, a_{n-1})$
nonzero, $\beta=1+\displaystyle \sum_{i=1}^{n-r-1}a_{i}^{2}$, $I_k\,
(n-r\leq k\leq n-1)$ are open intervals defined by the conditions
$|a_k\sqrt{\beta}x_k+b_k|<\pi/2$ while $I_n$ is defined by
$\left|-\dfrac{a_{n-r}\ldots a_{n-1}}{\sigma_{r-1}(a_{n-r},\ldots ,
a_{n-1})}\sqrt{\beta}x_n+b_n\right| < \pi/2$.
\end{theorem}

Recently, Seo \cite{Seo} gave a classification of the translation hypersurfaces
with constant mean curvature or constant Gauss-Kronecker curvature in Euclidean
space. Particularly, he proved that if $M$ is a translation hypersurface with
constant Gauss-Kronecker curvature $GK$ in $\Bbb{R}^{n+1}$, then $M$ is congruent
to a cylinder, and hence $GK=0$. In \cite{LSS}, Lima, Santos and Sousa proved
that given any integer $n\geq3$, any translation hypersurface in $\Bbb{R}^{n+1}$ with constant
scalar curvature must have zero scalar curvature. In this work, we generalize
these results to the curvatures $S_r$. Precisely, we prove the following

\begin{theorem}\label{th2}
Any translation hypersurface in $\Bbb{R}^{n+1}\, (n\geq3)$ with $S_r$ constant,
for $2<r<n$, must have $S_r=0$.
\end{theorem}

\section{Preliminaries and basic results}

Let $\overline{M}^{n+1}$ be a connected Riemannian manifold. In the remainder
of this paper, we will be concerned with isometric immersions, $\Psi:M^{n}\to
\overline{M}^{n+1}$, from a connected, $n$-dimensional orientable Riemannian
manifold, $M^{n}$, into $\overline{M}^{n+1}$. We fix an orientation of $M^{n}$,
by choosing a globally defined unit normal vector field, $N$, on $M$. Denote by $A$, the
corresponding shape operator. At each $p\in M$, $A$ restricts to a self-adjoint
linear map $A_{p}:T_{p}M\to T_{p}M$. For each $1\leq r\leq n$, let
$S_{r}:M^{n}\to\Bbb{R}$ be the smooth function such that $S_{r}(p)$ denotes the
$r$-th elementary symmetric function on the eigenvalues of $A_{p}$, which can be
defined by the identity
\begin{equation}\label{defSr}
 \det(A_p-\lambda I)=\sum_{k=0}^{n}(-1)^{n-k}S_{k}(p)\lambda^{n-k}.
\end{equation}
where $S_{0}=1$ by definition. If $p\in M^{n}$ and $\{e_{l}\}$ is a basis of
 $T_{p}M$, given by eigenvectors of $A_{p}$, with corresponding eigenvalues
$\{\lambda_{l}\}$, one immediately sees that
\[
 S_{r}=\sigma_{r}(\lambda_{1},\ldots,\lambda_{n}),
\]
where $\sigma_{r}\in\Bbb{R}[X_{1},\ldots,X_{n}]$ is the $r$-th elementary
symmetric polynomial on $X_{1},\ldots,X_{n}$. Consequently,
\[
S_{r}=\sum_{1\leq i_{1}<\cdots<i_{r}\leq n}\lambda_{i_{1}}\cdots\lambda_{i_{r}},
\, \, \, {\rm where} \, \, \, r=1,\ldots,n.
\]

In the next result we present an expression for the curvature $S_r$  of a
translation hypersurface in the Euclidean space. This expression will play an essential role in this paper.

\begin{proposition}\label{EqSr}
Let $F:\Omega\subset \Bbb{R}^n\to \Bbb{R}$ be a smooth function, defined
as follows $F(x_1,\ldots, x_n)=\sum_{i=1}^nf_i(x_i)$, where each $f_i$ is a
smooth function of one real variable. Let $M^n$ be the graphic of $F$, given
in coordinates by
\begin{equation}\label{paramet}
\varphi(x_1,\ldots,x_n)=\sum_{i=1}^n x_i e_i+F(x_1,\ldots,x_n)e_{n+1}
\end{equation}
The $S_r$ curvature of $M^n$ is given by
\begin{equation}\label{S_r}
S_r=\frac{1}{W^{r+2}}\cdot \sum_{1\leq i_1<\ldots <i_r\leq n}^n\ddot
f_{i_1}\ldots \ddot f_{i_r}(1+ \sum_{1\leq m\leq n\atop m\neq i_1\ldots i_r}
{\dot f}_m ^2),
\end{equation}
where the dot represents derivative with respect to the corresponding variable,
that is, for each $j=1,\ldots ,n$, one has $\dot f_j=\dfrac{df_j}{d
x_j}(x_j)=\dfrac{\partial F}{\partial x_j}(x_1,\ldots , x_n)$.
\end{proposition}
\begin{proof} Let $F$ be as stated in the Proposition, denote by $\nabla F=\sum_{i=1}^n\dfrac{\partial F}{\partial
x^i}\,\, e_i$ the Euclidean gradient of $F$ and $<,>$ the standard Euclidean inner
product. Then, we have
\begin{equation*}
\nabla F=\sum_{i=1}^n \dot f_i\,\, e_i
\end{equation*}
and the coordinate vector fields associated to the parametrization given
in (\ref{paramet}) have the following form
$$\frac{\partial \varphi}{\partial x_m}= e_m+\dot f_m e_{n+1},\quad m=1,\ldots, n.$$
Hence, the elements $G_{ij}$ of the metric of $M^n$ are given by
$$G_{ij}=\left< \frac{\partial \varphi}{\partial x^i}, \frac{\partial
\varphi}{\partial x^j}\right> =\delta_{ij}+\dot f_i\dot f_j,$$
implying that the matrix of the metric $G$ has the following form
\begin{equation*}\label{matriz_G}
G=I_n+(\nabla F)^t\,\, \nabla F,
\end{equation*}
where $I_n$ is the identity matrix of order $n$. Note that the $i$-th column of
$G$, which will be denoted by $G^i$ has the expression given by the column
vector
\begin{equation}\label{G^i}
G^i= e_i+\dot f_i \nabla F.
\end{equation}
An easy calculation shows that the unitary normal vector field $\xi$ of $M^n$
satisfies $$W\xi = e_{n+1}-\nabla F,$$ where $W^2=1+|\nabla F|^2$. Thus, the
second fundamental form $B_{ij}$ of $M^n$ satisfies
\begin{equation*}
WB_{ij}=\left< W\xi ,\frac{\partial^2\varphi}{\partial x^i\partial
x^j}\right>=\left< e_{n+1}-\nabla F, \delta_{ij}\ddot f_i
e_{n+1}\right>=\delta_{ij}\ddot f_i,
\end{equation*}
implying that the matrix of $B$ is diagonal
\begin{equation*}\label{matriz_B}
B=\frac{1}{W}\cdot \textrm{diag}(\ddot f_1,\ldots,\ddot f_n),
\end{equation*}
with $i$-th column given by the column vector
\begin{equation}\label{B^i}
B^i=\frac{\ddot f_i}{W}\,\,  e_i.
\end{equation}
If $A$ denotes the matrix of the Weingarten mapping, then $A=G^{-1} B$. In
(\ref{defSr}), changing $\lambda$ by $\lambda^{-1}$ gives
$$\det(\lambda A-I)=\sum_{i=1}^n(-1)^{n-i}S_i\lambda^{i}.$$
Thus, we conclude that the expression for curvature $S_r$ can be found by the
following calculation
\begin{equation*}
(-1)^{n-r}S_r=\frac{d^r}{d\lambda^r}_{\big|\lambda=0}\det(\lambda A-I).
\end{equation*}
Note that
\begin{equation*}
(-1)^{n-r}\det G\cdot S_r=\det G\cdot \frac{d^r}{d\lambda^r}_{\big|\lambda=0}
\det(\lambda A-I)=\frac{d^r}{d\lambda ^r}_{\big|\lambda=0}\det(\lambda B-G).
\end{equation*}
Due to the multilinearity of function $\det$, on its $n$ column vectors, it
follows immediately that
\begin{equation*}
\frac{d}{d\lambda}_{\big|\lambda=0}\det\ [\lambda B^1-G^1,\cdots, \lambda
B^n-G^n]=\sum_{i=1}^n(-1)^{n-1}\det\ [G^1,\cdots,\underbrace{B^i}_\textrm{i-th
term},\cdots,G^n],
\end{equation*}
leading to the conclusion
\begin{equation*}
\frac{d^r}{d\lambda^r}_{\big|\lambda=0}\det\ (\lambda B-G)  =  \sum_{1\leq
i_1<\ldots<i_r \leq n}^n(-1)^{n-r}\det\
[G^1,\cdots,B^{i_1},\cdots,B^{i_r},\cdots, G^n]
\end{equation*}
and thus
\begin{equation}\label{SrPARCIAL}
S_r=\frac{1}{\det G} \sum_{1\leq i_1<\cdots<i_r\leq n}^n\det\
[G^1,\cdots,B^{i_1},\cdots,B^{i_r},\cdots, G^n].
\end{equation}
Now, applying the expressions (\ref{G^i}) and (\ref{B^i}) in (\ref{SrPARCIAL})
we reach to the expression

\begin{equation}\label{S_r1}
S_r=\frac{1}{\det G\cdot W^r} \sum_{1\leq i_1<\ldots<i_r\leq n}^n
\ddot{f}_{i_1}\ldots \ddot{f}_{i_r}\det\ [ e_1+f'_1\nabla F,\ldots,
e_{i_1},\ldots, e_{i_r},\ldots,  e_n+f'_n\nabla F].
\end{equation}
Calculating the determinant on the right in the equality above, we get
\begin{eqnarray*}
\det\ [ e_1+\dot f_1\nabla F,&\ldots&,   e_{i_1},\ldots,e_{i_r}, \ldots ,
e_n+\dot f_n\nabla F]=\\&=&  1+\sum_{i\neq i_1,\ldots ,i_r} \dot f_i\det\ [
e_1,\ldots, e_{i_1},\ldots,\underbrace{\nabla F}_\textrm{$i$-th term},\ldots,
e_{i_r},\ldots,  e_n]\\ &=& 1+\sum_{1\leq i\leq n \atop{i \neq i_1,\ldots ,i_r
}} \dot f_i^2.
\end{eqnarray*}
Consequently, the expression for $S_r$ in (\ref{S_r1}) assumes the following
form
\begin{equation*}\label{S_r2}
S_r=\frac{1}{\det G\cdot W^r} \sum_{1\leq i_1<\ldots<i_r\leq n}^n
\ddot{f}_{i_1}\ldots \ddot{f}_{i_r}(1+\sum_{1\leq i\leq n \atop{i \neq
i_1,\ldots ,i_r }} \dot f_i^2).
\end{equation*}
Finally, using that $\det G=W^2$ we obtain the desired expression
\begin{equation*}\label{S_r3}
S_r=\frac{1}{W^{r+2}} \sum_{1\leq i_1<\ldots<i_r\leq n}^n\ddot{f}_{i_1}\ldots
\ddot{f}_{i_r}(1+\sum_{1\leq i\leq n \atop{i \neq i_1,\ldots ,i_r }} \dot
f_i^2).
\end{equation*}
\end{proof}

\section{Proof of the theorems}

In order to prove Theorem \ref{th1} we need the following lemma.

\begin{lemma}\label{casoparticular}
 Let $f_1,\ldots, f_r$ be smooth functions of one real variable satisfying the
differential equation
\begin{equation} \label{aux equat}
\sum_{k=1}^r\ddot{f_1}(x_1)\ldots \widehat{\ddot{f_k}(x_k)} \ldots
\ddot{f_r}(x_r)(\beta+\dot{f_k}^{2}(x_k))=0,
\end{equation}
where $\beta$ is a positive real constant and the big hat means an omitted term.
If $\ddot{f_i}\ne0$, for each $i=1, \ldots r$ then

 \begin{equation*} \label{solution partial}
\sum_{k=1}^{r} f_k(x_k) = \sum_{k=1}^{r-1}\frac{\sqrt{\beta}}{a_k}\ln
\left|\dfrac{\cos\left(-\dfrac{a_1\ldots a_{r-1}}{\sigma_{r-2}(a_1,\ldots ,
a_{r-1})}\sqrt{\beta}x_r+b_r\right)}{\cos(a_k\sqrt{\beta}x_k+b_k)}\right|+c
 \end{equation*}
where $a_i,b_i,c$, $i=1,\ldots r$ are real constants with $a_i ,
\sigma_{r-2}(a_1,\ldots , a_{r-1})\neq 0$.
\end{lemma}
\begin{proof}
Since the derivatives $\ddot f_i\ne 0$ it follows that $\ddot{f_1}(x_1) \ldots
\ddot{f_r}(x_n)\ne 0$. Thus dividing (\ref{aux equat}) by this product we get
the equivalent equation:

\begin{equation*}
\sum_{k=1}^r\dfrac {\beta+\dot{f_k}^{2}(x_k)}{\ddot{f_k}(x_k)}=0,
\end{equation*}
 which implies, after taking derivative with respect to $x_l$ for each
$l=1,\ldots r$, that $\left(\dfrac
{\beta+\dot{f_l}^{2}(x_l)}{\ddot{f_l}(x_l)}\right)'=0$, thus $\dfrac
{\beta+\dot{f_l}^{2}(x_l)}{\ddot{f_l}(x_l)}= \tilde a_l$ for some non null
constant $\tilde a_l$. Thus, setting $a_l=\dfrac1{ \tilde a_l}$
 \begin{equation*}
 \dfrac {\ddot{f_l}(x_l)}{\beta+\dot{f_l}^{2}(x_l)}=a_l \quad \mbox{for each
$l=1,\ldots ,r$}
 \end{equation*}
 which can be easily solved to give:
 \begin{equation*}
  \arctan\left(\frac{\dot{f_l}(x_l)}{\sqrt{\beta}}\right) = a_l\sqrt{\beta}x+b_l
\quad\mbox{for some constant $b_l$}
 \end{equation*}
 and consequently
 \begin{equation} \label{fl}
 f_l(x_l) = -\frac{1}{a_l}\sqrt{\beta}\ln|\cos(a_l\sqrt{\beta}x_l+b_l)|+c_l,
\quad l=1,\ldots , r.
 \end{equation}
Now, since $\sum_{k=1}^r \dfrac 1{a_k}=0$ it implies that
$\dfrac1{a_r}=-\dfrac{\sigma_{r-2}(a_1,\ldots , a_{r-1})}{a_1\ldots a_{r-1}}$,
from (\ref{fl}) it follows that
 \begin{equation*}
 f_r(x_r) = \dfrac{\sigma_{r-2}(a_1,\ldots , a_{r-1})}{a_1\ldots
a_{r-1}}\sqrt{\beta}\ln|\cos(a_r\sqrt{\beta}x_r+b_r)|+c_r
 \end{equation*}
Consequently
 \begin{equation*}
\sum_{k=1}^{r} f_k(x_k) = \sum_{k=1}^{r-1}\frac{1}{a_k}\sqrt{\beta}\ln
\left|\dfrac{\cos\left(-\dfrac{a_1\ldots a_{r-1}}{\sigma_{r-2}(a_1,\ldots ,
a_{r-1})}\sqrt{\beta}x_r+b_r\right)}{\cos(a_k\sqrt{\beta}x_k+b_k)}\right|+c,
 \end{equation*}
where $c=c_1+\ldots+c_r$.
\end{proof}

With this lemma at hand we can go to the proof of Theorem \ref{th1}.


\begin{proof}[\sc{Proof of Theorem \ref{th1}}] From Proposition \ref{EqSr},
we have that $M^{n}$ has zero $S_r$ curvature if, and only if,
\begin{equation}\label{EqSrzero}
\sum_{1\leq i_1<\ldots <i_r \leq n}\ddot{f}_{i_1} \ldots \ddot{f}_{i_r}
\Big(1+\sum_{ {1\leq k\leq n} \atop {k\notin \{i_1,\ldots i_r\}}}
\dot{f}_{k}^{2}\Big)=0.
\end{equation}

\noindent In order to ease the analysis, we divide the proof in four cases.

\vspace{.3cm}

\noindent\textbf{Case\, 1:} Suppose $\ddot{f}_{i}(x_{i})=0$, $\forall\,
i=1,\ldots, n-r+1$. In this case, we have no restrictions on the functions
$f_{n-r+2}, \ldots , f_{n}$. Thus
\[
 \Psi(x_{1},\ldots,x_{n})=(x_{1},\ldots,x_{n},\sum_{i=1}^{n-r+1}a_{i}x_{i}
+\sum_{j=n-r+2}^{n}f_{j}(x_{j})+b)
\]
where $a_{i},b\in\Bbb{R}$ and for $l=n-r+2,\ldots , n$, the functions $f_{l}:I_l\subset\Bbb{R}\to\Bbb{R}$ are arbitrary smooth
functions of one real variable. Note that the parametrization obtained comprise
hyperplanes.

\vspace{.3cm}

\noindent\textbf{Case\, 2:} Suppose $\ddot{f}_{i}(x_{i})=0$, $\forall\,
i=1,\ldots, n-r$, then, there are constants $\alpha_{i}$ such  that
$\dot f_{i}=\alpha_{i}$, for $i=1,\ldots, n-r$. From (\ref{EqSrzero}) we have
\[
\ddot{f}_{n-r+1}\ldots \ddot{f}_{n}(1+\alpha_{1}^{2}+\cdots+\alpha_{n-r}^{2})=0,
\]
from which we conclude that $\ddot{f}_{k}=0$ for some $k\in \{ n-r+1, \ldots
n\}$
and thus, this case is contained in the Case $1$.

\vspace{.3cm}

\noindent\textbf{Case\, 3:} Now suppose $\ddot{f}_{i}(x_{i})=0$, $\forall\,
i=1,\ldots, n-r-1$ and $\ddot{f}_{k}(x_{k})\ne0$, for every $k=n- r, \ldots ,n$.
Observe that
if we had $\ddot{f}_{k}(x_{k})=0$ for some $k=n-r, \ldots, n$ the analysis
would reduce to the Cases 1 and 2. In this case,  there are constants
$\alpha_{i}$
such that $\dot f_{i}=\alpha_{i}$ for any $1\leq i\leq n-r-1$. From
(\ref{EqSrzero})
we have
\[
\sum_{k=n-r}^n \ddot{f}_{n-r}\ldots \widehat{\ddot{f}_{k}} \ldots \ddot{f}_n
(\beta+\dot{f}_{k}^{2})
=0
\]
where $\beta=1+\displaystyle\sum_{k=1}^{n-r-1}\alpha_{k}^{2}$ and the hat means
an omitted term. Then, from Lemma
\ref{casoparticular} we have that
 \begin{equation*} \label{solution partial}
\sum_{k=n-r}^{n} f_k(x_k) = \sum_{k=n-r}^{n-1}\frac{\sqrt{\beta}}{a_k}\ln
\left|\dfrac{\cos\left(-\dfrac{a_{n-r}\ldots
a_{n-1}}{\sigma_{r-1}(a_{n-r},\ldots ,
a_{n-1})}\sqrt{\beta}x_n+b_n\right)}{\cos(a_k\sqrt{\beta}x_k+b_k)}\right|+c
 \end{equation*}
where $a_{n-r},\ldots ,a_{n-1},b_{n-r},\ldots , b_{n}$ and $c$ are real
constants, and $a_{n-r}, \ldots, a_{n-1}$, and $\sigma_{r-1}(a_{n-r}, \ldots,
a_{n-1})$ are nonzero.

\vspace{.3cm}

\noindent\textbf{Case\, 4:} Finally, suppose $\ddot{f}_{i}(x_{i})=0$, where
$1\leq i\leq k$ and $n-k\geq r+2$, and $\ddot{f}_{i}(x_{i})\ne 0$ for any $i>k$.
We will show that this case cannot occur. In fact, note that for any fixed $l\ge
k+1$
\begin{eqnarray*}
 \sum_{k+1\leq i_1<\ldots < i_r\leq n}\ddot{f}_{i_1}\ldots \ddot{f}_{i_r}
\Big(1&+&\sum_{ { {1\leq m\leq n} \atop {m\not=i_1,\ldots i_r}} }
\dot{f}_{m}^{2}\Big)\\
&=&\ddot{f}_{l}\sum_{{k+1\leq i_1< \ldots <i_{r-1} \leq n \atop i_1,\ldots
,i_{r-1}\not= l}}\ddot{f}_{i_1} \ldots \ddot{f}_{i_{r-1}}
\Big(1+\sum_{ { {1\leq m\leq n} \atop {m\not=l,i_1,\ldots,i_{r-1}}} }
\dot{f}_{m}^{2}\Big)\\
 &+&\sum_{{k+1\leq i_1<\ldots i_r\leq n \atop i_1,\ldots ,i_r\not=l} }
\ddot{f}_{i_1}\ldots \ddot{f}_{i_r} \Big(1+\sum_{ { {1\leq m\leq n} \atop
{m\not=i_1,\ldots i_r}} } \dot{f}_{m}^{2}\Big)\\
\end{eqnarray*}

\noindent Derivative with respect to the variable $x_{l}\, (l\geq k+1)$, in the
above equality, gives
\begin{eqnarray}\label{derivadal}
 \dddot{f}_{l}\sum_{{k+1\leq i_1<\ldots<i_{r-1}\leq n \atop i_1,\ldots,i_{r-1}
\not=l}} \ddot{f}_{i_1}\ldots\ddot{f}_{i_{r-1}}\Big(1&+&\sum_{ {
{1\leq m\leq n} \atop {m\not=l,i_1,\ldots,i_{r-1}}} }
\dot{f}_{m}^{2}\Big)\nonumber\\
&+&2\dot{f}_{l}\ddot{f}_{l}
\sum_{k+1\leq i_1<\ldots<i_r\leq n \atop i_1,\ldots,i_r\not=l}
\ddot{f}_{i_1}\ldots\ddot{f}_{i_r}=0.
\end{eqnarray}

\noindent That is, if we set
\begin{eqnarray*}
A_l&=&\sum_{{k+1\leq i_1<\ldots<i_{r-1}\leq n \atop i_1,\ldots,i_{r-1}
\not=l}} \ddot{f}_{i_1}\ldots\ddot{f}_{i_{r-1}}\Big(1+\sum_{ {
{1\leq m\leq n} \atop {m\not=l,i_1,\ldots,i_{r-1}}} }
\dot{f}_{m}^{2}\Big) \quad\mbox{and}\quad\\
 B_l&=& \sum_{k+1\leq i_1<\ldots<i_r\leq n \atop i_1,\ldots,i_r\not=l}
\ddot{f}_{i_1}\ldots\ddot{f}_{i_r}
\end{eqnarray*}
then, it follows that $A_l, B_l$ do not depend on the variable $x_l$ and we can
write
\begin{equation} \label{AlBl}
A_l\dddot f_l+2B_l\dot f_l\ddot f_l=0.
\end{equation}
We have two possible situations to take into account: {\bf Case I.} $A_l\neq 0,\, \forall\,l\geq k+1$, and {\bf Case II.}
there is an $l\ge 1$ such that $A_l=0$.

\vspace{.2cm}

\noindent\textbf{Case I.\, $A_{l}\neq 0$:} Under this assumption, there are
constants $\alpha_{l}\, (l=k+1,\ldots,n)$ such that equation (\ref{AlBl})
becomes $\dddot{f}_{l}+2\alpha_{l}\dot{f}_{l}\ddot{f}_{l}=0$. Furthermore,
it can be shown that for $\{l_1,\ldots,l_{r+1}\} \subset \{k+1,\ldots,n\}$
\begin{equation}\label{dGr}
\frac{\partial^{r+1}G_r(f_1,\ldots,f_n)}{\partial x_{l_1}\cdots \partial
x_{l_{r+1}}} =2\sum_{k=1}^{r+1}\left(\dot{f}_{l_k}\,\ddot{f}_{l_k} \prod_{m=1
\atop m\ne k}^{r+1}\dddot{f}_{l_m} \right)
\end{equation}
where
\begin{equation*}
G_r(f_{k+1},\ldots,f_n):=W^{r+2}S_r=\displaystyle \sum_{k+1\leq
i_1<\ldots<i_r\leq n}^n\ddot{f}_{i_1}\ldots \ddot{f}_{i_r}(1+\sum_{1\leq i\leq n
\atop{i \neq i_1,\ldots ,i_r }} \dot f_i^2).
\end{equation*}
Since $S_r = 0$ it follows that $G_r = 0$, and using that $\displaystyle
\prod_{k=1}^{r+1} \dot{f}_{l_k}\,\ddot{f}_{l_k}\ne0$ we obtain
\begin{eqnarray}\label{aa}
\sum_{k=1}^{r+1}\left(\prod_{m=1 \atop m\ne k}^{r+1}
\frac{\dddot{f}_{l_m}}{\dot{f}_{l_m}\ddot{f}_{l_m}}\right)&=&\frac{ \displaystyle
\sum_{s=1}^{r+1} \left(\dot{f}_{l_s}\,\ddot{f}_{l_s} \prod_{m=1 \atop m\ne
s}^{r+1}\dddot{f}_{l_m} \right)}{\displaystyle \prod_{k=1}^{r+1}
\dot{f}_{l_k}\,\ddot{f}_{l_k}}=0.
\end{eqnarray}
Now, for $l=l_1,\ldots l_{r+1}$, substitute $\dddot{f}_{l}+2\alpha_{l} \dot{f}_{l}\ddot{f}_{l}=0$ in (\ref{aa}) to
obtain the identity
\begin{equation}\label{Srlr+1}
\sigma_{r}(\alpha_{l_1},\ldots,\alpha_{l_r},\alpha_{l_{r+1}})=0
\end{equation}
for any $l_1,\ldots,l_r,l_{r+1}\in\{k+1,\ldots,n\}$. Hence we conclude that,
\begin{eqnarray*}
 \sigma_{r}(\alpha_{k+1},\ldots,\alpha_{n})&=&0\\
\sigma_{r+1}(\alpha_{k+1},\ldots,\alpha_{n})&=&0.
\end{eqnarray*}
These equalities, from Proposition \ref{Caminha}, imply that at most $r-1$
of the constants $\alpha_{l}\, (l\geq k+1)$ are nonzero. If  $\alpha_{l_1}\ne0,
\ldots,\alpha_{l_{m}}\ne0$ with $m\leq r-1$, in the expression obtained
for $B_l$, making $l\ne l_1,\ldots,l_{m}$ and taking derivatives with respect
to the variables $x_{l_1},\ldots,x_{l_{m}}$ we get
\[
\prod_{j=l_1}^{l_{m}}\dddot{f}_j\cdot\sigma_{r-m}(\ddot{f}_{k+1},\ldots,
\widehat{\ddot{f}_l},\ldots,\widehat{\ddot{f}}_{l_1},\ldots,
\widehat{\ddot{f}}_{l_{m}},\ldots,\ddot{f}_n)=0
\]
for all $l\in\{k+1,\ldots,n\}\smallsetminus\{l_1,\ldots,l_{m}\}$. As
$\dddot{f}_j\ne0$ for all $j\in\{l_1\ldots,l_{m}\}$, we obtain that
\[
\sigma_{r-m}(\ddot{f}_{k+1},\ldots,\widehat{\ddot{f}}_l,\ldots,\widehat
{\ddot{f}}_{l_1},\ldots,\widehat{\ddot{f}}_{l_{m}},\ldots,\ddot{f}_n)=0
\]
for all $l\in\{k+1,\ldots,n\}\smallsetminus\{l_1,\ldots,l_{m}\}$. Consequently,
\begin{eqnarray*}
\sigma_{r-m}(\ddot{f}_{k+1},,\ldots,\widehat{\ddot{f}}_{l_1},\ldots,\widehat
{\ddot{f}}_{l_{m}},\ldots,,\ddot{f}_{n})&=&0\\
\sigma_{r-m+1}(\ddot{f}_{k+1},,\ldots,\widehat{\ddot{f}}_{l_1},\ldots,\widehat
{\ddot{f}}_{l_{m}},\ldots,,\ddot{f}_{n})&=&0.
\end{eqnarray*}
Since $(n-k-m)-(r-m)=n-k-r\geq2$, at most $r-m-1$ of the functions
$\ddot{f}_{l}$ are nonzero, for $k+1\leq l\leq n$ and $l\ne l_1,\ldots,l_m$, leading to a contradiction. So, $\alpha_j=0$ for all
$j\in\{l_1\ldots,l_{r-1}\}$, which implies that $\ddot{f}_{l}$ is
constant for all $l\in\{k+1,\ldots,n\}$. Now, again from equation
(\ref{derivadal}) we get
\[
\sum_{k+1\leq i_1<\ldots<i_r\leq n \atop i_1,\ldots,i_r\not=l}
\ddot{f}_{i_1}\ldots\ddot{f}_{i_r}=0, \qquad \mbox{for any $l\in\{k+1,\ldots,n\}$}.
\]
\noindent From which, we conclude that
\begin{eqnarray*}
\sigma_{r}(\ddot{f}_{k+1},\ldots,\ddot{f}_{n})&=&0\\
\sigma_{r+1}(\ddot{f}_{k+1},\ldots,\ddot{f}_{n})&=&0.
\end{eqnarray*}

\noindent Therefore, at most $r-1$ of the functions $\ddot{f}_{l}\,\,(k+1\leq l\leq
n)$ are nonzero, leading to a contradiction. Thus, it follows that Case 4 cannot
occur, if $A_l\neq 0$ for every $l$.

\vspace{.3cm}

\noindent\textbf{Case\, $A_{l} = 0$:} In this case, we have $B_l\dot{f_l}
\ddot{f_l}=0$ implying
\begin{eqnarray*}
A_l&=&\sum_{{k+1\leq i_1<\ldots<i_{r-1}\leq n \atop i_1,\ldots,i_{r-1}
\not=l}} \ddot{f}_{i_1}\ldots\ddot{f}_{i_{r-1}}\Big(1+\sum_{ {
{1\leq m\leq n} \atop {m\not=l,i_1,\ldots,i_{r-1}}} }
\dot{f}_{m}^{2}\Big)=0 \quad\mbox{and}\quad\\
 B_l&=& \sum_{k+1\leq i_1<\ldots<i_r\leq n \atop i_1,\ldots,i_r\not=l}
\ddot{f}_{i_1}\ldots\ddot{f}_{i_r}=0.
\end{eqnarray*}

\noindent Derivative of $A_l$ with respect to variable $x_s$, for
$s=k+1,\ldots,n$ and $s\neq l$, gives

\begin{eqnarray}\label{dAls}
&&\dddot{f}_{s}\sum_{{k+1\leq i_1<\ldots<i_{r-2}\leq n \atop
i_1,\ldots,i_{r-2} \not=l,s}} \ddot{f}_{i_1}\ldots\ddot{f}_{i_{r-2}}
\Big(1+\sum_{ { {k+1\leq m\leq n} \atop {m\not=l,s,i_1,\ldots,i_{r-2}}}
}\dot{f}_{m}^{2}\Big)\nonumber\\
&&\qquad +2\dot{f}_{s}\ddot{f}_{s}\sum_{{k+1\leq i_1<\ldots<i_{r-1}\leq
n \atop i_1,\ldots,i_{r-1}\not=l,s}}\ddot{f}_{i_1}\ldots
\ddot{f}_{i_{r-1}}=0.
\end{eqnarray}

\noindent Now, for $i_1,\ldots,i_r\in\{k+1,\ldots ,n \}$ with $i_1,\ldots,i_r,l$
distinct indices, taking the derivatives of $B_{l}$ with respect to $x_{i_1},
\ldots,x_{i_r}$ gives
\begin{eqnarray*}
\dddot{f}_{i_1}\ldots\dddot{f}_{i_r}=0.
\end{eqnarray*}
Consequently, for at most $r-1$ indices, say $i_1,\ldots,i_{r-1}$, we can have
$\dddot f_{i_m}\neq 0$, $(m=1,\ldots,r-1)$, and $\dddot f_j=0$ for every
$j=k+1,\ldots ,n$, with $j\neq l,i_1,\ldots,i_{r-1}$. Thus $\dddot f_{i_m}\neq
0$, with $i_m\neq l$, together with equation $\frac{\partial B_l}{\partial
x_{i_m}} =0$ implies that the sum
\[
\sum_{k+1\leq i_1<\ldots<i_{r-1}\leq n \atop i_1,\ldots,i_{r-1}\not=l,i_m}
\ddot{f}_{i_1}\ldots\ddot{f}_{i_{r-1}}=0.
\]
Now, if $\dddot{f}_j = 0$ we have by equation (\ref{dAls}) that
\[
\sum_{k+1\leq i_1<\ldots<i_{r-1}\leq n \atop i_1,\ldots,i_{r-1}\not=l,j}
\ddot{f}_{i_1}\ldots\ddot{f}_{i_{r-1}}=0.
\]
Therefore,
\[
\sum_{k+1\leq i_1<\ldots<i_{r-1}\leq n \atop i_1,\ldots,i_{r-1}\not=l,j}
\ddot{f}_{i_1}\ldots\ddot{f}_{i_{r-1}}=0, \quad j=k+1,\ldots ,n
\,\, \mbox{and $j\neq l$}
\]
\noindent From which, we conclude that
\begin{eqnarray*}
\sigma_{r-1}(\ddot{f}_{k+1},\ldots,\widehat{\ddot{f}}_l,\ldots,\ddot{f}_{n})
&=&0\\
\sigma_{r}(\ddot{f}_{k+1},\ldots,\widehat{\ddot{f}}_l,\ldots,\ddot{f}_{n})&=&0.
\end{eqnarray*}
Thus, for at most $r-2\,(r\geq3)$ indices we must have $\ddot f_j\neq0$, for every
$j=k+1,\ldots ,n$, and $j\neq l$.  This contradicts the hypothesis assumed in
Case 4. Hence, $A_l=0$ cannot occur. Since the case $A_l\neq 0$, cannot occur as
well, it follows that Case 4 is not possible. This completes the proof of the
theorem.
\end{proof}


\vspace{.1cm}

\begin{proof}[\sc{Proof of Theorem \ref{th2}}] Let $M^n\subset\Bbb{R}^{n+1}$ be
a translation hypersurface with constant $S_{r}$ curvature. First, note that
\begin{eqnarray}\label{d^dW^r}
\frac{\partial^mW^{r+2}}{\partial x_{i_1}\cdots \partial x_{i_m}}
&=&\prod_{j=1}^{m}(r+4-2j)\cdot\prod_{k=1}^{m}\dot{f}_k\,\ddot{f}_k\cdot
W^{r+2-2m}.
\end{eqnarray}
We have as a consequence of the proof of Theorem 1, see (\ref{dGr}), the identity
\[
\frac{\partial^{r+1}G_r(f_1,\ldots,f_n)}{\partial x_{l_1}\cdots \partial
x_{l_{r+1}}} =2\sum_{k=1}^{r+1}\left(\dot{f}_{l_k}\,\ddot{f}_{l_k} \prod_{m=1
\atop m\ne k}^{r+1}\dddot{f}_{l_m} \right)
\]
where $G_r(f_1,\ldots,f_n)=\displaystyle \sum_{1\leq i_1<\ldots<i_r\leq
n}^n\ddot{f}_{i_1}\ldots \ddot{f}_{i_r}(1+\sum_{1\leq i\leq n \atop{i \neq
i_1,\ldots ,i_r }} \dot f_i^2)$. 
With this we conclude, by Proposition \ref{EqSr}, that

\begin{eqnarray}\label{dWr}
\prod_{j=1}^{r+1}(r+4-2j)\cdot\prod_{k=1}^{r+1}\dot{f}_{l_k}\,\ddot{f}_{l_k}\cdot
W^{-r}S_r&=&\frac{\partial^{r+1}\,(W^{r+2}S_r)}{\partial x_{l_1}\cdots \partial
x_{l_{r+1}}}\nonumber\\
&=&2\sum_{k=1}^{r+1}\left(\dot{f}_{l_k}\,\ddot{f}_{l_k} \prod_{m=1
\atop m\ne k}^{r+1}\dddot{f}_{l_m} \right).
\end{eqnarray}
Now, we have two cases to consider: $r$ odd and $r$ even.

\vspace{.2cm}

\noindent\textbf{Case\, $r$ odd:} Suppose that there are $l_1,\ldots, l_{r+1}$
such that $\displaystyle\prod_{k=1}^{r+1} \dot{f}_{l_k}\,\ddot{f}_{l_k}\ne0$.
Then,
\begin{eqnarray*}
Q_r:=\prod_{j=1}^{r+1}(r+4-2j)\cdot W^{-r}S_r&=&2\frac{\displaystyle
\sum_{s=1}^{r+1} \left(\dot{f} _{l_s}\,\ddot{f} _{l_s} \prod_{m=1 \atop m\ne
s}^{r+1}\dddot{f}_{l_m} \right)}{\displaystyle \prod_{k=1}^{r+1}
\dot{f}_{l_k}\,\ddot{f}_{l_k}}\\
&=&2\sum_{k=1}^{r+1}\left(\prod_{m=1 \atop m\ne
k}^{r+1}\frac{\dddot{f}_{l_m}}{\dot{f}_{l_m}\ddot{f}_{l_m}}\right).
\end{eqnarray*}
Therefore,
\[
\frac{\partial^{r+1}Q_r}{\partial x_{l_1}\cdots \partial
x_{l_{r+1}}}=0.
\]
On the other hand, using (\ref{d^dW^r}) we obtain
\[
\frac{\partial^{r+1}Q_r}{\partial x_{l_1}\cdots \partial x_{l_{r+1}}}
=\displaystyle\prod_{j=1}^{r+1}(r+4-2j)\prod_{i=1}^{r+1}(-r+2-2i)
\prod_{k=1}^{r+1} \dot{f}_{l_k}\,\ddot{f}_{l_k}W^{-3r-2}S_r.
\]
Since $r$ is odd, we conclude that $r+4-2j\ne0$ and $-r+2-2j\ne0$, for any $j\in\Bbb{N}$ and, therefore, $S_r=0$.

\vspace{.1cm}

\noindent Now, if for at most $r$ indices we
have $\ddot f_j\neq0$ for example $j=l_1,\ldots,l_r$ then
\[
W^{r+2}S_r=\ddot{f}_{l_1}\cdots\ddot{f}_{l_r}\alpha,
\]
for some constant $\alpha\ne0$. Thus,
\[
(r+2)W^r\dot{f}_{l_1}\ddot{f}_{l_1}S_r=\dddot{f}_{l_1}\ddot{f}_{l_2}
\cdots\ddot{f}_{l_r}\alpha.
\]
If $\dddot{f}_{l_1}=0$, then $S_r=0$. Otherwise,
\[
(r+2)W^{r+2}\dot{f}_{l_1}\ddot{f}_{l_1}S_r=\dddot{f}_{l_1}\ddot{f}_{l_2}
\cdots\ddot{f}_{l_r}W^2\alpha\quad \Rightarrow \quad (r+2)\dot{f}_{l_1}(\ddot{f}_{l_1})^2
=\dddot{f}_{l_1}W^2.
\]
As $r>1$ implies that $W$ does not depend on the variables $x_{l_2},\ldots, x_{l_n}$,
it follows that $\ddot{f}_{l_2}=\cdots=\ddot{f}_{l_n}=0$ leading to a contradiction.

\vspace{.2cm}

\noindent\textbf{Case\, $r$ even:} In this case, there is a natural $q\geq2$
such that $r=2q$. Then $r+1\geq q+2$ and consequently
\[
\prod_{k=1}^{r+1}(r+4-2k)=0.
\]
Therefore, by (\ref{dWr}) we get
\[
\sum_{k=1}^{r+1}\left(\dot{f}_{l_k}\,\ddot{f} _{l_k} \prod_{m=1
\atop m\ne k}^{r+1}\dddot{f}_{l_m} \right)=0.
\]
Suppose that there are $l_1,\ldots, l_{r+1}$ such that $\displaystyle
\prod_{k=1}^{r+1}\ddot{f}_{l_k}\ne0$. In this case,
\[
\sum_{k=1}^{r+1}\left(\prod_{m=1 \atop m\ne
k}^{r+1}\frac{\dddot{f}_{l_m}}{\dot{f}_{l_m}\ddot{f}_{l_m}}\right)=0.
\]
We conclude that for each $l_i$ there is a constant $\alpha_{l_i}$
such that $\dddot{f}_{l_i}=\alpha_{l_i}\dot{f}_{l_i}\ddot{f}_{l_i}$.
Now, it is easy to verify (see (\ref{derivadal})) that
\begin{eqnarray*}
(r+2)\dot{f}_{l_{r+1}}\ddot{f}_{l_{r+1}}W^rS_r&=&\frac{\partial\,
G_r(f_1,\ldots,f_n)}{\partial x_{l_{r+1}}}\\&=&\dddot{f}_{l_{r+1}}\,
G_{r-1}(f_1,\ldots,\widehat{f}_l,\ldots,f_n)\\&+&2\dot{f}_{l_{r+1}}
\ddot{f}_{l_{r+1}}\sum_{1\leq i_1<\ldots<i_r\leq n \atop
i_1,\ldots,i_r\not=l_{r+1}}\ddot{f}_{i_1}\ldots\ddot{f}_{i_r}\,.
\end{eqnarray*}
Therefore,
\[
(r+2)W^rS_r=\alpha_{l_{r+1}}\,G_{r-1}(f_1,\ldots,\widehat{f}_{l_{r+1}},\ldots,
f_n)\\+2\sum_{1\leq i_1<\ldots<i_r\leq n \atop i_1,\ldots,i_r\not=l_{r+1}}
\ddot{f}_{i_1}\ldots\ddot{f}_{i_r}\,.
\]
Differentiating this identity with respect to the variable $x_{l_{r+1}}$, gives
\[
(r+2)\,r\,\dot{f}_{l_{r+1}}\ddot{f}_{l_{r+1}}W^{r-2}S_r=0\quad \mbox{implying that} \quad S_r=0.
\]
Finally, suppose that for any $(r+1)$-tuple of indices, say $l_1,\ldots, l_{r+1}$ it holds that
$\displaystyle\prod_{k=1}^{r+1}\ddot{f}_{l_k}=0$. Then,
\begin{eqnarray*}
\sigma_{r+1}(\ddot{f}_{1},\ldots,\ddot{f}_{n})&=&0\\
\sigma_{r+2}(\ddot{f}_{1},\ldots,\ddot{f}_{n})&=&0.
\end{eqnarray*}
Implying that at least $n-r$ derivatives $\ddot{f}_{l}$ vanish, i.e.,
there are at most $r$ functions such that $\ddot f_j\neq0$ for example
$j=l_1,\ldots,l_r$. Thus, by Proposition \ref{EqSr}
\[
W^{r+2}S_r=\ddot{f}_{l_1}\cdots\ddot{f}_{l_r}\alpha
\]
for some constant $\alpha\ne0$. We conclude that $S_r=0$ analogously
to the way it was presented for the case $r$ odd.
\end{proof}

\section{Final remarks}

The problem of classifying translation surfaces under curvature constraints is far from been complete.
It is remarkable that the solutions give rise to generalized Scherk-like solutions and it would be interesting
to understand which solutions could be obtained when considering other classes of curvature.
Another problem to be considered consists in classify the translation n-surfaces of higher codimension
or imbedded into other Riemann or semi-Riemannian manifolds.



\begin{thebibliography}{20}
\bibitem{Caminha} A. Caminha: \textit{On spacelike hypersurfaces of constant
sectional curvature lorentz manifolds}, J. of Geometry and Physics, \textbf{56}
(2006), 1144--1174.

\bibitem{CST} C. Chen, H. Sun and L. Tang: \textit{On translation Hypersurfaces
with Constant Mean Curvature in $(n+1)$-Dimensional Spaces}, J. Beijing Inst.
of Tech., \textbf{12} ($n.3$) (2003), 322--325.

\bibitem{Dillen} F. Dillen, L. Verstraelen and G. Zafindratafa: {\em A
generalization of the translational surfaces of Scherk}, Differential Geometry
in honor of Radu Rosca, Catholic University of Leuven - Belgium (1991),
107--109.

\bibitem{MLeite} M. L. Leite: \textit{An example of a triply periodic complete
embedded scalar-flat hypersurface of $\Bbb{R}^{4}$}, Anais da Academia Brasileira
de Ci\^{e}ncias, \textbf{63} (1991), 383--385.

\bibitem{LSS} B. P. Lima, N. L. Santos and P. A. Sousa: \textit{Translation
Hypersurfaces with Constant Scalar Curvature into the Euclidean Space}, Israel
J. Math., to appear in Israel J. Math..

\bibitem{Liu} H. Liu: \textit{Translation Surfaces with Constant Mean Curvature
in 3-Dimensional Spaces}, J. Geom. \textbf{64} (1999), 141--149.

\bibitem{Lopez} R. L\'{o}pez: \textit{Minimal translation surfaces in
hyperbolic space}, Beitr. Algebra Geom. \textbf{52} (2011), 105--112.

\bibitem{Scherk} H. F. Scherk: \textit{Bemerkungen $\ddot{u}$ber die kleinste
Fl$\ddot{a}$che innerhalb gegebener Grenzen}, J. Reine Angew. Math. \textbf{13}
(1835), 185--208.

\bibitem{Seo} K. Seo: \textit{Translation Hypersurfaces  with Constant
Curvature in Space Forms}, Osaka J. Math., \textbf{50} (2013), 631--641.

\end{thebibliography}
\end{document}